\documentclass[11pt,a4paper]{article}

\usepackage[latin1]{inputenc}

\usepackage[normalem]{ulem}

\usepackage{amssymb}

\usepackage{graphicx}

\usepackage{graphicx,indentfirst,amsmath,amsfonts,amssymb,amsthm,newlfont}
\usepackage{epsfig}
\newtheorem{theorem}{Theorem}[section]
\newtheorem{proposition}[theorem]{Proposition}

\newtheorem{definition}[theorem]{Definition}

\newtheorem{corollary}[theorem]{Corollary}

\newcommand{\dif}{\mathrm{Diff}}
\newcommand{\difH}{\mathrm{Diff}(\Delta^H, M)}
\newcommand{\difV}{\mathrm{Diff}(\Delta^V, M)}

\newcommand{\R}{\mathbf R}
\newcommand{\N}{\mathbf N}
\newcommand{\Z}{\mathbf Z}

\begin{document}
\begin{center}


 {\Large {\bf Decomposition of stochastic flows generated \\[3mm] by Stratonovich SDEs
with jumps 
 }}

\end{center}

\vspace{0.1cm}

\begin{center}
{\large Alison M. Melo\footnote{E-mail: alison.melo@univasf.edu.br. 
Research
supported by CNPq 142084/2012-3.} \ \ \ \ \ \ \ \ \ \ \ \ \
 { Leandro Morgado}\footnote{E-mail:
leandro.morgado@ufsc.br.
Research supported by  FAPESP 11/14797-2.}

\medskip

{ Paulo R. Ruffino}\footnote{Corresponding author, e-mail:
ruffino@ime.unicamp.br.
Partially supported by FAPESP  12/18780-0,
11/50151-0 and CNPq 477861/2013-0.}}

\vspace{0.2cm}

\textit{Departamento de Matem\'{a}tica, Universidade Estadual de Campinas, \\
13.083-859- Campinas - SP, Brazil.}

\vspace{4mm}

\end{center}

\begin{abstract}

Consider a manifold $M$ endowed locally with
a pair of complementary distributions  $\Delta^H \oplus \Delta^V=TM$ and let
$\text{Diff}(\Delta^H,
M)$ and $\text{Diff}(\Delta^V, M)$ be the corresponding Lie subgroups
generated by vector fields in the
corresponding distributions. We decompose a stochastic flow with jumps, up
to a
stopping time, as $\varphi_t = \xi_t \circ \psi_t$, where $\xi_t \in
\text{Diff}(\Delta^H, M)$ and $\psi_t \in \text{Diff}(\Delta^V, M)$. Our main
result provides Stratonovich stochastic differential equations with jumps
for each of these two components in the corresponding infinite dimensional Lie
groups. We present an extension of the Itô-Ventzel-Kunita formula for
stochastic flows with jumps generated by classical Marcus equation (as in
Kurtz, Pardoux and Protter \cite{KPP}). The results here
correspond to an
extension of
Catuogno, da Silva and Ruffino
\cite{CSR2}, where this decomposition was studied for the continuous case.


\end{abstract}


\vspace{0.3cm}
\noindent {\bf MSC2010 subject classification:} 60H10, 58J65, 58D05.

\section{Introduction}

Let $M$ be a compact differentiable manifold, with $\text{Diff}(M)$ the
corresponding infinite dimensional Lie group of diffeomorphisms of M generated
by the Lie algebra of smooth vector fields. Consider a stochastic flow
$\varphi_t$ of local diffeomorphisms in $M$, generated by a Stratonovich SDE
driven by a semimartingale $Z_t$
with jumps. The question we address in this paper is the possibility of
decomposing $\varphi_t$ in components that belong to specific subgroups of
$\text{Diff(M)}$, which provide dynamical or geometrical information of the
system. In the continuous case, decompositions of this kind for stochastic
flows has been studied in the literature, with many distinct frameworks and
with different aimed subgroups. Among others see  Bismut \cite{Bismut}, Kunita
\cite{Kunita-1}, \cite{Kunita-2},  Ming Liao \cite{ML} and some of our previous
work \cite{CSR}, \cite{Colonius-Ruffino}, \cite{Ruffino}. In the last few papers
mentioned, for example, geometrical conditions on a Riemannian manifold have
been stated to guarantee the existence of the decomposition where the first
component lies in the subgroup of isometries or affine transformations.

Our geometrical framework here is the following: consider a
manifold $M$ equipped with
a pair of complementary distributions  $\Delta^H$ and
$\Delta^V$ (horizontal and vertical, respectively) in the sense
of differentiable sections in  Grassmanian bundles of $M$. 
For each smooth 
vector field $X$ in $M$, denote by $\exp \{ tX\}\in \dif(M)$ the
associated flow.  Given a
distribution $\Delta$ in $M$, we shall denote by
$\dif(\Delta, M) $ the closure of the subgroup of diffeomorphisms
generated by exponentials
of vector fields in $\Delta$, precisely:
\begin{equation*}
\dif(\Delta, M)= \mathrm{cl} \Big\{  \exp \{t_1X_1\}\circ \ldots
\circ \exp \{t_nX_n\}, \mbox{\  with }  X_i \in
\Delta, 
t_i \in \R,  \mbox{ for all } n\in \N \Big\}.
\end{equation*}
The differentiable
structure of infinite dimensional Lie group of diffemorphisms of a compact
manifold $M$  is classically given
by maps in the locally convex space of smooth vector fields in $M$, more
precisely, maps in an ILB-space, see  Omori \cite{Omori}, Milnor
\cite{Milnor} and
more recently Neeb \cite{Neeb}. So that the closure above is taken with respect
to this topology.
Existence and uniqueness of solution of an ODE in an ILB-spaces with smooth
vector fields has being guaranteed by \cite[Thm
7.2, Chap.1]{Omori}. Hence, solutions of Stratonovich differential equations in
these spaces can be obtained via support theorem if the noise is finite
dimensional, see e.g. Twardowska \cite{Twardowska}. 
Note that, considering
the limits, by Baker-Campbell-Hausdorff
formula, this  subgroup includes flows of diffeomorphisms associated to $\mathrm{Lie}(\Delta)$, the Lie algebra generated by $\Delta$. Precisely, we have that  
\[
 \mathrm{Diff} (\Delta, M)= \mathrm{Diff}(\mathrm{Lie}( \Delta), M).
\]

Using this notation, the stochastic flow $\varphi_t$ belongs to
 $\mathrm{Diff}(TM,
M)$, which is a connected subgroup of $\mathrm{Diff}(M)$. Again,
$\mathrm{Diff}(TM,
M)$ contains the Lie subgroups $\difH$
and
$\difV$. 
If both distributions $\Delta^H$ and $\Delta^V$ are involutive, locally
the intersection of these two subgroups reduces to the identity, and the 
elements of
each of these
subgroups send leaves of the corresponding foliation into themselves. This also 
implies that if the decomposition exists, it is unique.

In this article we want to decompose a stochastic flow, up to a
stopping time, as $\varphi_t = \xi_t \circ \psi_t$, where $\xi_t \in
\text{Diff}(\Delta^H, M)$ and $\psi_t \in \text{Diff}(\Delta^V, M)$. Our main
result provides Stratonovich stochastic differential equations with jumps
for each of these two components in the corresponding infinite dimensional Lie
groups (Theorem \ref{thm: decomposicao}).
One of the key points on
obtaining the equations for each
component $\xi_t$ and $\psi_t$ is an
extension of the classical Itô-Ventzel-Kunita formula for discontinuous càdlàg
processes, which are also generated by SDEs with jumps (Theorem \ref{thm:
Ito-Ventzel}); here, we generalize the definition of the integration
with respect to a change of variables in Kurtz, Pardoux and Protter
\cite{KPP}.

A motivation for this decomposition appears in many  dynamical systems. For
instance, in a foliated space (say, assuming that the horizontal  distribution
is integrable) the decomposition
measures how far the trajectories are from being foliation-preserving:
namely,
the closer the second component $\psi_t$ is to the identity, closer the system
is to be (horizontal) foliation-preserving. Vice-versa, changing the order of
the vertical and horizontal distributions and the order of the components in
the decomposition, the same interpretation holds if the vertical component is
integrable. For example: a
principal bundle with a connection is a natural state
space with this feature, where the vertical distribution is determined by
the fibres and the horizontal distribution is given by the geometry. The
dynamics in averaging principles for transversal
perturbation in foliated systems also illustrates the relevance of this
decomposition, see e.g. the Liouville torus in Hamiltonian
systems in Li \cite{Li}, or in a general compact foliation in Gargate and Ruffino
\cite{Gargate-Ruffino}.

In this paper, we consider the Stratonovich SDEs driven by semimartingales with
jumps $Z_t$, in the sense of the so called Marcus equation (see e.g.
\cite{KPP}). This means, among
other aspects (recalled in Section 2), that the jumps of the solution occur in
the direction of the deterministic flow generated by the corresponding vector
fields. This property  guarantees that if the manifold $M$ is embedded in
an Euclidean space and the vector fields are in $TM$ then, after
the jumps, the solution flow remains in $M$. Also, we assume
that the number of jumps in $Z_t$ is finite in a bounded interval of time
a.s..
Hence, it includes as a possible integrator the so called L\'evy--jump diffusion
(see Applebaum \cite{ap}), but not L\'evy process in general. This idea of
finite jumps in bounded intervals has parallel in the theory of chain control
sets, see Colonius and Kliemann
\cite{Colonius-Kliemann}, Patr\~ao and San Martin \cite{sanmartin}, and references therein.

The paper is organized as follows: In Section 2, we describe some basic
properties and definitions of Stratonovich SDE driven by a semimartingale with
jumps, in the sense of Marcus equation. In Section 3, our
extension of the Itô-Ventzel-Kunita formula is presented, for two stochastic
flows generated by Marcus equations with respect to the same semimartingale
with jumps. Finally, in section 4, we deal with the
decomposition of the stochastic flow
$\varphi_t= \xi_t \circ \psi_t$,  obtaining the corresponding SDEs for each
component. The results here correspond to an extension of the article Catuogno,
da Silva and Ruffino
\cite{CSR2}, where the same decomposition is studied for continuous flows.

\section{Stratonovich SDEs with jumps}

For reader's convenience, we recall the definition of a Stratonovich SDE driven
by a semimartingale with jumps in $\R^d$, in the sense of Marcus equation (for
more
details, see e.g. Kurtz, Pardoux e Protter \cite{KPP}):

\begin{equation}\begin{array}{lll} \label{eq: define marcus}
dx_t = \displaystyle \sum_{i=1}^m X^i(x_t) \diamond dZ^i_t,
\end{array}
\end{equation}
with initial condition $x(0) = x_0$. The solution, in its integral
form is:
\begin{equation}\begin{array}{lll}
x_t  = x_0 + \displaystyle \sum_{i=1}^m \int_0^t X^i(X_s) \diamond dZ^i_s =:x_0
+ \displaystyle \int_0^t X(x_s) \diamond dZ_s,
\end{array}
\end{equation}
where $x_t$ is an adapted stochastic process taking values in $\R^d$, the
integrator $\{Z^i_s : s \geq 0\}$  is a semimartingale with jumps and $X^i$ are
vector fields in $\R^d$ for all $i \in \{1,\ldots,m\}$. For sake of
simplicity we use the matrix notation $X=(X^1, \ldots , X^m)$ and $Z=(Z^1,
\ldots , Z^m)$. The solution is
interpreted as a stochastic process that satisfies the equation:

\begin{equation}\begin{array}{lll} \label{eq: interpreta marcus}
x_t = x_0 + \displaystyle \int_0^t X(x_{s_-}) \ dZ_s \ + \ \frac{1}{2}
\displaystyle \int_0^t X'X(x_s) \ d[Z,Z]_s^c \  \vspace{0.2cm}\\
\hspace{1cm} + \displaystyle \sum_{0 < s \leq t} \left\{ \phi (X \Delta Z_s,
x_{s_-}) - x_{s_-} - X(x_{s_-}) \Delta Z_s \right\},
\end{array}
\end{equation}
in the following sense: the first term on the right hand side of $(\ref{eq:
interpreta marcus} )$ is a standard Itô integral of the predictable process
$X(x_{t_-})$ with respect to the semimartingale $Z_t$. The second term is a
Stieltjes integral with respect to the continuous part of the quadratic
variation of $Z_t$. In the third term:  $\phi (X \Delta Z_s, x_{s_-})$
indicates the solution at time one of the ODE generated by the vector field $X
\Delta Z_s$ and initial condition $x_{s_-}$. Thus, the
jumps of this equation occurs in deterministic directions. Regularity
conditions on vector fields $X$ implies that there exists a unique
stochastic flow of diffeomorphisms $\varphi_t$. Conditions on the derivatives of
the vector
fields
guarantee that there exists a stochastic flow of local diffeomorphisms.
Moreover, for an embedded submanifold $M$ in an Euclidean space, if the vector
fields of the equation (\ref{eq: define marcus}) are in $TM$ then, for each
initial condition on $M$, the solution stays in $M$ a.s., see \cite{KPP}.

Let $Y$ be smooth vector fields in $\R^d$, and consider $G_t$ a
$d$-dimensional process given by
$ d G_t = Y (G_t) \diamond d Z_t$. If $H \in C^1(\R^d; \R ^m)$) we recall the
definition of the
Stratonovich integral of $H (G_t)$ with respect to $Z_t$ (Definition 4.1 in
\cite{KPP}):
\begin{equation}
\begin{array}{lll} \label{eq: integral KPP}
\displaystyle \int_0^t H(G_s) \diamond d Z_s := \displaystyle \int_0^{t}
H(G_{s-}) d Z_s + \frac{1}{2} \ \mathrm{Tr} \displaystyle \int_0^{t} H' (G_s) d
[Z,Z]^c_s \ Y(G_s)^t \\
\hspace{2.6cm} + \displaystyle \sum_{0 < s \leq t} \left( \displaystyle
\int_0^1 \left[
H(\phi(Y \Delta Z_s, G_{s-}, u)) - H(G_{s-}) \right] du \right) \Delta Z_s,
\end{array}
\end{equation}
where $\phi(Y \Delta Z_s, G_{s-}, u)$ has the same interpretation as before:
it is the solution of the ODE generated by the vector field $Y \Delta Z_s$, and
initial condition $G_{s-}$, at time $t = u$. More specifically,
given $t_0 > 0$, the jump of the integral on the left hand side of equation
(\ref{eq: integral KPP}) is given by:
$$
\Delta \left( \displaystyle \int_0^t H(G_s) \diamond d Z_s \right)_{
\hspace{-1mm}t_0} =
\left( \displaystyle \int_0^1 H(\phi(Y \Delta Z_{t_0}, G_{t_{0-}}, u)) du
\right)
\Delta Z_{t_0}.
$$

The dynamical interpretation of the expression above is that, after opening a
unit interval with a  `fictitious curve' which connects $G_{t_{0-}}$ and
$G_{t_0}$, the jump of the integral is given by the mean of $H$ along this
curve multiplied by the jump of the semimartingale $Z_t$, see \cite{KPP}.

We  consider that the number of jumps in the integrator process $Z_t$
is finite in bounded intervals a.s.. For example, this hypothesis includes the
so called Levy-jump diffusion
$
Z^i_t =  B^i_t +  \sum_{k=0}^{N_t} J^i_k,
$
where $B^i_t$ is a Brownian motion, $N_t$ is a Poisson process and the random
variables $(J_k)$ are  i.i.d. (see e.g.
Applebaum \cite{ap}).

\section{An extension of Itô-Ventzel-Kunita formula}

In the continuous case, we have the well known Itô-Ventzel-Kunita formula:
Let $\varphi_t$ be a stochastic flow of (local) diffeomorphisms, which is
solution of the following Stratonovich SDE
$$dx_t = \displaystyle \sum_{i=1}^m X^i(x_t) \circ dN^i_t,$$
where $N_i$ are continuous semimartingales, and $X^i$ are smooth vector fields
in $\R^d$, for $1 \leq i\leq m $. Let $U_t = (U^1_t, 
\ldots, U^d_t)$
be a continuous semimartingale. Then:

\begin{theorem} (Itô-Ventzel-Kunita formula for continuous case) \label{Thm: Ito-Ventzel cont}
\begin{equation}\begin{array}{lll}
\varphi_t (U_t) = \varphi_0 (U_0) + \displaystyle \sum_{i=1}^m \displaystyle
\int_0^t X^i (\varphi_s(U_s)) \circ dN^i_s + \displaystyle \sum_{j=1}^d 
\displaystyle
\int_0^t \frac {\partial \varphi_s}{\partial x_j} (U_s) \circ d U^j_s.
\end{array}
\end{equation}
\end{theorem}

For a proof, see e.g. Kunita \cite[Thm 8.3]{Kunita-2}). A direct
corollary of this result is a Leibniz formula for the
composition of stochastic flows of diffeomorphisms in the continuous case,
that is, $d(\varphi \circ F)_t = d \varphi_t \circ F_t + (\varphi_t)_*
\circ d F_t$, where the differential form is considered in the  Stratonovich
sense.

In the proof of the main result of the next section, one needs an extension of
this Leibniz formula for flows generated by Stratonovich SDE driven by
semimartingales with jumps. Precisely, let $\psi_t$ and $\xi_t$ be flows of
diffeomorphisms
generated by Marcus equations with respect to the same semimartingale with
jumps $Z_t$:
$$\diamond d \psi_t = X(\psi_t) \diamond d Z_t \ \ \ \text{and} \ \ \ \diamond
d \xi_t = Y (\xi_t) \diamond d Z_t.$$
where $X, Y$ are smooth vectors fields. We assume that $Z_t$ has the property
of finite jumps, as described in the previous section.  We
define the
following integral which generalizes the classical Marcus
integral ($\ref{eq: interpreta marcus}$), (i.e. \cite[Eq.
2.2]{KPP}):

\begin{equation}\begin{array}{lll} \label{eq: definicao}
\displaystyle \int_0^{t} {\psi_s}_*    Y (\xi_s) \diamond dZ_s &  \hspace{-2mm}
:= \hspace{-2mm} & \displaystyle
\int_0^{t} {\psi_s}_*  Y (\xi_{s^{-}})d Z_s 
 +  \displaystyle \frac{1}{2}
\int_0^{t} \big( X' (Y(\xi_s) ) + \psi_{s*} (Y'Y) \big)\, d\, [Z,Z]_s^c
 \vspace{3mm}\\
&& + \displaystyle \sum_{0 \leq s \leq t} \bigg\{ - {\psi_{s^-}}_* (Y(\xi_{s^{-}}))  \Delta Z_s  - \phi (X \Delta Z_{t_0},
\psi_{t_0^-}(\xi_{t_0^-}))
\vspace{0.3cm} \\
&& \hspace{34mm} +\  \phi \big( X
\Delta Z_s, \psi_{s^{-}} (\phi (Y \Delta
Z_s,  \xi_{s-})) \big)  \bigg\}
\end{array}
\end{equation}

The first term on the right hand side of (\ref{eq: definicao}) is the standard
Itô integral of the predictable process ${\psi_s}_* (Y(\xi_{s^{-}}))$. The
second term corresponds to the finite variation, such that the continuous part
of the expression satisfies the classical Itô-Ventzel-Kunita formula. In the summation, we use
the notation $\phi (X, x_0, u)$ for the solution of ordinary differential
equation with respect to the vector field $X$, and initial condition $x_0$ at
time $u$ (if omitted, it means that $u=1$). Note that if $\psi_s$ is the
identity for all $s\in [0,t]$, one recovers equation (\ref{eq: interpreta
marcus}).

With the same notation as before, our extension of
Itô-Ventzel-Kunita formula is given by

\begin{theorem} \label{thm: Ito-Ventzel} (Itô-Ventzel-Kunita for Stratonovich
SDE with jumps) Suppose that the stochastic flows $\psi_t$ and $\xi_t$ are
defined
in the interval $[0,a]$. Then, for all $t \in [0, a]$:
\begin{equation}\begin{array}{lll} \label{eq: integral generalizada}
\psi_t (\xi_t) = \psi_0 (\xi_0) + \displaystyle \int_0^t X(\psi_s(\xi_s)) \diamond d
Z_s + \displaystyle \int_0^{t} {\psi_s}_* (Y(\xi_s)) \diamond d Z_s.
\end{array}
\end{equation}

\begin{proof}
Consider 
$
 t_0 = \sup \ \{ s\in [0,a] \mbox{ such that the formula holds in } [0,s] \}.
$
We want to prove that $t_0 = a$. The proof follows by contradiction considering two cases: when the driven
process $Z_t$ is continuous at $t_0$, and when it jumps  at $t_0$. In the first
case, if $t_0 < a$, there exists a.s. an $\varepsilon > 0$ such that the process $Z_t$ is
continuous in $[t_0, t_0 + \varepsilon)$. So, by
It\^o-Ventzel formula for the continuous case (Thm. \ref{Thm: Ito-Ventzel cont}), the relation is valid up to time $t_0 +
\varepsilon$, hence we have a contradiction. 

On the other hand, if  the process $Z_t$ jumps at $t_0$, and $t_0 < a$, 
we prove that the relation is also valid for $t_0+ \varepsilon$, for a positive
$\varepsilon$. In fact, concatenating the jump of $\xi_{t_0}$ and the jump of $\psi_{t_0}$ we have that:
\begin{equation} \label{Eq: soma e subtrai}
\psi_{t_0} (\xi_{t_0})  =  \psi_{t_0^-} (\xi_{t_0^-}) + \phi \big(X \Delta
Z_s, \psi_{t_0^-} (\phi(Y \Delta Z_s, (\xi_{t_0^-})) \big) - \psi_{t_0^-}
(\xi_{t_0^-}). 
\end{equation}
Initially use that 
\[
\psi_{t_0^-} (\xi_{t_0^-})= \psi_0 (\xi_0) + \displaystyle \int_0^{t_0^-} X(\psi_s(\xi_s)) \diamond d Z_s +
\displaystyle \int_0^{t_0^-} {\psi_s}_* (Y(\xi_s)) \diamond d Z_s. 
\]
By definition,  
\begin{equation*}\begin{array}{lll}
\displaystyle \int_0^{t_0} X(\psi_s(\xi_s)) \diamond d Z_s &=& \displaystyle \int_0^{t_0^-} X(\psi_s(\xi_s)) \diamond d Z_s + \phi (X \Delta Z_{t_0},
\psi_{t_0^-}(\xi_{t_0^-})) - \psi_{t_0^-}(\xi_{t_0^-})   
\end{array}
\end{equation*}
and 
\begin{equation*}\begin{array}{lll}
\displaystyle \int_0^{t_0} {\psi_s}_* (Y(\xi_s)) \diamond d Z_s &=& \displaystyle \int_0^{t_0^-} {\psi_s}_* (Y(\xi_s)) \diamond d Z_s - \phi (X \Delta Z_{t_0},
\psi_{t_0^-}(\xi_{t_0^-})) \vspace{4mm}\\
 && +\phi \big(X \Delta
Z_{t_0}, \psi_{t_0^-} (\phi(Y \Delta Z_{t_{0}}, (\xi_{t_0^-}))) \big). 
\end{array}
\end{equation*}
Hence,  equation (\ref{Eq: soma e subtrai})  yields 
\begin{equation*}\begin{array}{lll}
\psi_{t_0} (\xi_{t_0})&=& \psi_0 (\xi_0) + \displaystyle \int_0^{t_0} X(\psi_s(\xi_s)) \diamond d Z_s +
\displaystyle \int_0^{t_0} {\psi_s}_* (Y(\xi_s)) \diamond d Z_s.
\end{array}
\end{equation*}

So, the formula also holds for $t = t_0$. By assumption, there exists a. 
s. an $\varepsilon >
0$ such that $Z_t$ is continuous in $(t_0, t_0 + \varepsilon)$. Applying 
again the classical Itô-Ventzel-Kunita formula, we have a
contradiction. Hence $t_0 = a$.

\end{proof}

\end{theorem}

Next corollary contains the Leibniz formula which we are going to use in order  to obtain 
the SDEs for the components of our decomposition in each corresponding Lie 
subgroup. Its proof follows straightforward from Theorem 3.2. 

\begin{corollary}
\label{Leibniz}
 Let $\psi_t$ and $\xi_t$ be flows generated by Marcus equations with
respect
to the same semimartingale $Z_t$, with the property of finite jumps, then:
\begin{equation}\begin{array}{lll} \label{eq: Liebniz generalizada}
\diamond d(\psi \circ \xi)_t = \diamond d \psi_t \circ \xi_t + (\psi_t)_*
\circ \diamond d \xi_t.
\end{array}
\end{equation}
\end{corollary}

\section{Decomposition of stochastic flows with jumps}

In this section, following the framework of \cite{KPP}, we assume
that the manifold $M$ is embedded in an Euclidean space. Hence, smooth vector
fields in $M$ can be extended to smooth vector fields in a tubular
neighbourhood of $M$. Let $\varphi_t$ be
a stochastic
flow of (local) diffeomorphisms in $M$ generated by a Stratonovich SDE in the
sense of Marcus equation, where
the integrator has the property of finite jumps, as described in the
previous section.

Assume that the manifold $M$ has (locally) a pair of differentiable
distributions, that we call horizontal and vertical. We denote the first one by
$\Delta^H: U \subseteq M \to Gr_k(M)$ and the second by $\Delta^V: U \subseteq
M \to Gr_{n-k}(M)$. Also, we assume that the horizontal and vertical
distributions are complementary, that is, for each $x \in U$, $\Delta^H(x)
\oplus \Delta^V(x) = T_x M$.

%

\subsection{Existence of the decomposition}

We are interested in decomposing the stochastic flow as $\varphi_t = \xi_t 
\circ \psi_t$, where the components $\xi_t \in \text{Diff}(\Delta^H, M)$ and 
$\psi_t\in \text{Diff}(\Delta^V, M)$. Initially, consider the following 
definitions:

\begin{definition}[Adjoint distribution]
Let $\Delta: U \subseteq M \to Gr_{p}(M)$ be a differential distribution in
$M$, and take $\xi \in \text{Diff}(M)$. The distribution
$Ad(\xi)\Delta: U \subseteq M \to Gr_{p}(M)$ is given by:
$$Ad(\xi)\Delta (x) = \xi_* \Delta (\xi^{-1}(x)).$$
\end{definition}

\begin{definition}
\label{Def: transversality}
Let $\Delta^H$ and $\Delta^V$ be a pair of locally complementary distributions
in $M$. We say that $\Delta^H$ and $\Delta^V$ preserve transversality along
$\text{Diff}(\Delta^H, M)$ if, for each $\xi \in \text{Diff}(\Delta^H, M)$, the
distributions $\Delta^H$ and $Ad(\xi) \Delta^V$ are locally complementary.
\end{definition}

Assuming the condition of transversality in the distributions, we have a
result on the decomposition of a continuous stochastic flow
${\varphi}_t$ generated by a canonical Stratonovich SDE \cite[Thm 
2.2]{CSR2}:

\begin{theorem}
\label{Thm: continuous}
Let $\Delta^H$ and $\Delta^V$ be two complementary distributions
in $M$ which preserve transversality along $\text{Diff}(\Delta^H, M)$. Given a
continuous stochastic flow ${\varphi}_t$ then, up to a stopping time, there
exists a factorization ${\varphi}_t = {\xi}_t \, \circ\,
{\psi}_t$, where ${\xi}_t$ is a continuous diffusion in
$\text{Diff}(\Delta^H, M)$ and ${\psi}_t$ is a continuous process in
$\text{Diff}(\Delta^V, M)$.
\end{theorem}

Initially, a straightforward extension of this result holds for  a stochastic 
flow $\varphi_t$ generated by a Stratonovich SDE with respect to a 
semimartingale $Z_t$ with jumps (proposition below). In 
the next section we present a stronger result showing the equations of each 
component in $\difH$ and $\difV$.

\begin{proposition} \label{prop: existence}
Suppose that the distributions $\Delta^V$ and $\Delta^H$ preserves
transversality along $\text{Diff}(\Delta^H, M)$. Given a stochastic 
flow $\varphi_t$ generated by a Stratonovich SDE with jumps, then, up to a 
stopping time $\tau$,
there exists a factorization 
$\varphi_t = \xi_t \circ \psi_t$,
where $\xi_t$  is a diffusion in
$\text{Diff}(\Delta^H, M)$ and ${\psi}_t$ is a  process in
$\text{Diff}(\Delta^V, M)$.

Moreover, if  $t_0$ is a point of discontinuity of $Z_t$, $t_0 \leq \tau$ 
and the decomposition holds for $\varphi_{t_0}$, then $t_0 < \tau$, a.s..

\begin{proof} By assumption, the first discontinuity of the integrator $Z_t$ 
happens at a stopping time $t_1>0$ a.s.. Hence, Theorem \ref{Thm: continuous} 
guarantees the existence of the decomposition at least up to $t_1$.
For the last statement,
take $\delta > 0$ such that the integrator $Z_t$ is continuous in
$(t_0, t_0 + \delta)$
and apply again Theorem \ref{Thm: continuous}.
\end{proof}
\end{proposition}

In the last statement of the Proposition above, the
hypothesis on the decomposition of the diffeomorphism $\varphi_{t_0}$ cannot
be removed: In fact, it may happen that the decomposition
holds up to the time $t_0^-$, and the process jumps to a non decomposable
diffeomorphism $\varphi_{t_0}$. As a basic example to illustrate this fact,
consider the canonical horizontal and vertical foliations in $R^2$. Take the
pure rotation system  $dx_t = A x_t \diamond dZ_t$, where $Z_t$ is a
semimartingale such that $Z_t \notin \{\pi/2 + k\pi, k \in \Z
\} $ for $t<t_0$. The corresponding flow and its
decomposition for $t< t_0$ is given by

$$\varphi_t = \left(
                                \begin{array}{cc}
                                  \cos Z_t & -\sin Z_t \\
                                  \sin Z_t & \cos Z_t
                                \end{array}
                              \right)=  \left(
                                \begin{array}{cc}
                                  \sec Z_t & -\tan Z_t \\
                                  0 & 1
                                \end{array}
                              \right)
                              \left(
                                \begin{array}{cc}
                                  1 & 0 \\
                                  \sin Z_t & \cos Z_t
                                \end{array}
                              \right).$$

If at time $t_0$ the process jumps to $Z_{t_0} \in \{\pi/2 + k\pi, k \in \Z \}$,
then $\varphi_{t_0}$ is not decomposable. On the other hand, if the stochastic
flow jumps to a decomposable diffeomorphism $\varphi_{t_0}$, then the
decomposition holds at least an additional  $\delta>0$ longer than $t_0$.
In Morgado and Ruffino \cite{MR}, we have used this fact to construct a
stochastic flow $\widetilde{\varphi_t}$ generated by a Stratonovich SDE with
controlled jumps that approximates arbitrarily a stochastic flow
${\varphi}_t$ generated by a continuous  Stratonovich SDE, with the
property that $\widetilde{\varphi_t}$ can be decomposable for all $t>0$.

\subsection{Stratonovich SDEs with jumps for the components}

Using the extension of Itô-Ventzel-Kunita formula for Stratonovich SDE with jumps,
one can obtain the correspondent SDEs (in the sense of Marcus equation) for the components $\xi_t$ and $\psi_t$ in
the decomposition above. Initially, we have that $\varphi_t$ satisfies the
following right invariant equation in the Lie group terminology:
$$
d\varphi_t = \displaystyle \sum_{i=1}^m R_{\varphi_t *} X^i \diamond dZ^i_t,
$$
where $R_{\varphi_t *}$ is the derivative of the right translation of $\varphi_t$ in the
identity of $\text{Diff(M)}$. In fact, for all $i \in \{ 1, \ldots, m\}$, it 
holds that $R_{\varphi_t *} X^i = X^i (\varphi_t)$, since if $\gamma$ is a curve 
in $\text{Diff(M)}$ such that $\gamma(0) = Id$ and $\gamma'(0) = X^i$, we have:
$$R_{\varphi_t*} (X^i) = \frac{d}{ds} \big(\gamma(s) \circ \varphi_t \big)
\big{|}_{s=0} = X^i(\varphi_t).$$


By assumption in Proposition \ref{prop: existence}, we consider that $\Delta^H$ 
and $\Delta^V$ preserves transversality along $\text{Diff}(\Delta^H, M)$. So, 
for each $\xi \in \text{Diff}(\Delta^H,M)$, the distributions $\Delta^H$ and 
$Ad(\xi)\Delta^V$ are complementary. For each $x \in U$, we have a 
unique decomposition of $X^i(x)$ in the corresponding subspaces, given by:
$$
X_i(x) = \widetilde{X_i}(x) + v_i (\xi,x),
$$
where $\widetilde{X_i}(x) \in \Delta^H(x)$ and $v_i (\xi,x) \in
Ad(\xi)\Delta^V (x)$. Now, consider the autonomous Stratonovich SDE in the subgroup
$\text{Diff}(\Delta^H,M)$, given by:
$$
dx_t = \displaystyle \sum_{i=1}^m \widetilde{X_i}(x_t) \diamond dZ_t^i,
$$
with initial condition $x_0 = Id$. We define $\xi_t$ as the solution of this
equation. So, in the Lie group terminology, the
stochastic flow $\xi_t$ satisfies:
\begin{equation}\begin{array}{lll}
\label{eq: xi}
d\xi_t = \displaystyle \sum_{i=1}^m R_{\xi_t *} \widetilde{X_i} \diamond dZ_t^i.
\end{array}
\end{equation}
Note that this equation is not right invariant in the Lie group since 
$\widetilde{X_i}$ depends on $\xi_t$. By support theorem, it follows that 
$\xi_t \in \text{Diff}(\Delta^H,M)$. Finally, define $\psi_t = \xi_t^{-1} \circ 
\varphi_t.$ 
Using the Corollary \ref{Leibniz}, we can write the following equation for 
$d\xi_t^{-1}$:
\begin{equation}\begin{array}{lll}
d\xi_t^{-1} = \displaystyle \sum_{i=1}^m -L_{\xi_t^{-1} *} \widetilde{X_i}
\diamond dZ_t^i,
\end{array}
\end{equation}
where $L_{\xi_t^{-1} *}$ is the derivative of the left translation of
$\xi_t^{-1}$ at the
identity of the Lie group $\mathrm{ Diff(\Delta^H, M)}$. In fact, just take the 
time derivative of $\xi_t \circ \xi_t^{-1} = Id$. Therefore, using again the 
Corollary \ref{Leibniz}, one can obtain an expression for $d \psi_t$, given by:
\begin{equation}\begin{array}{lll}
\label{eq: psi}
d\psi_t &=& \displaystyle \sum_{i=1}^m \Big( \xi_t^{-1} X_i \ \xi_t \ \psi_t -
\xi_t^{-1} \widetilde{X_i} \ \xi_t \ \psi_t \Big) \diamond dZ_t^i = \\
&=& \displaystyle \sum_{i=1}^m Ad(\xi_t^{-1}) \ \big(X_i - \widetilde{X_i}\big)
\ \psi_t \diamond dZ_t^i = \\
&=& \displaystyle \sum_{i=1}^m Ad(\xi_t^{-1}) \ v_i (\xi_t) \ \psi_t \diamond
dZ_t^i.
\end{array}
\end{equation}

Note that, by construction, we have that $v_i(\xi_t, x) \in Ad(\xi_t) 
\Delta^V(x)$. Hence, 
$Ad(\xi_t^{-1})v_i(\xi_t,x) \in \Delta^V(x)$ for all $x \in U$, and using
support theorem again, one conclude that $\psi_t \in \text{Diff}(\Delta^V,M)$.
So, with the same notation as before, we have proved the following
result:

\begin{theorem} \label{thm: decomposicao}
The components $\xi_t$ and $\psi_t$ in the decomposition above can be
described, respectively, as the solutions of the following Stratonovich SDEs
with jumps:
$$
d \xi_t = \displaystyle \sum_{i=1}^m \widetilde{X_i}(\xi_t) \diamond dZ_t^i \ \ 
\
\ \mbox{ and }\ \  d \psi_t = \displaystyle \sum_{i=1}^m Ad(\xi_t^{-1}) \ v_i 
(\xi_t) (\psi_t)
\diamond dZ_t^i.
$$
with initial conditions $ \xi_0= \psi_0=Id$.
\end{theorem}

\subsection{Compatibility of the vector fields and the
pair of distributions}

In some
geometrical structures where we have a sort of compatibility of the vector
fields with
the pair of complementary distributions, it is possible to guarantee that the
decomposition holds for all positive time $t\geq 0$.

Given an initial condition $x_0\in M$, consider a local coordinate system
 $\alpha_0: U_0\subset M \rightarrow  \R^{p} \times \R^{n-p}$, with
$U_0$ an open set of $M$ containing $x_0$ and such that $\alpha_*
(\Delta^V(x_0))= \{0\} \times \R^{n-p}  $. Analogously, at each time $t\geq 0$,
consider $\alpha_t: U_t\subset M \rightarrow  \R^{p} \times \R^{n-p}$,
with $U_t$ an
open set of $M$ containing $\varphi_t(x_0)$ such that $\alpha_{t*}
(\Delta^V(\varphi_t(x_0))= \{0\} \times \R^{n-p}$.
With respect to these coordinate systems one writes the original flow
as $\varphi_t = (\varphi_t^1 (x,y), \varphi_t^2 (x,y))$.

\begin{proposition} Suppose that the decomposition  $\varphi_t= \xi_t \circ
\psi_t$ holds up to the stopping  time $\tau$ in a neighbourhood of the
initial condition $x_0$. Writing in coordinates,

\[
 \tau = \sup \{t >0;  \det
\frac{\partial \varphi_s^2(x,y)}{\partial y}\neq 0 \mbox{ for all } 0\leq s\leq t\}.
\]
\end{proposition}

\begin{proof} Straightforward
by the inverse function theorem: The local decomposition exists if and only if
the $(n-p)\times (n-p)$ submatrix $\frac{\partial \varphi^2(x,y)}{\partial y}$ 
is
invertible. That is, $ \varphi = (\xi^1 (x,y), Id_2) \circ (Id_1, \varphi^2
(x,y))$, where $Id_1$ and $Id_2$ are the identities in $\R^p$ and in
$\R^{n-p}$,
respectively.

\end{proof}

The proposition above says that the
hypothesis for the decomposition can be weakened to conditions only on
the original flow
$\varphi_t$ instead of using the transversality preserving property (Definition
\ref{Def:
transversality}).

\begin{corollary} Suppose that the vertical distribution is integrable
(a foliation). If the the original flow $\varphi_t$ sends each
vertical
leaf into a vertical leaf, then the decomposition $\varphi_t= \xi_t
\circ \Psi_t$ holds for all time $t\geq 0$.
\end{corollary}

\begin{proof}
Straightforward from the proposition above since, for all $t\geq 0$,
\[
 \varphi_{t*}= \left( \begin{array}{cc}
\frac{\partial \varphi_t^1(x,y)}{\partial x} & 0 \\
 & \\
\frac{\partial \varphi_t^1(x,y)}{\partial x} & \frac{\partial
\varphi_t^2(x,y)}{\partial y}
               \end{array} \right),
\]
hence
 $\det
\frac{\partial \varphi_t^2(x,y)}{\partial y}\neq 0$ for all $t\geq 0$.

\end{proof}

\noindent {\bf Example 1.} Consider in $\R^n\setminus
\{0\}$ the distribution $\Delta^H$ given by the spherical foliation and
the distribution $\Delta^V$ given by the radial foliation. Then, for this
pair of distribution, any linear  system $\varphi_t$ has the decomposition
$\varphi_t=\xi_t \circ \psi_t$
for all $t\geq 0$ since each radial line is sent to a radial line.

\bigskip

\noindent {\bf Example 2.} Most of relevant dynamics in fibre bundles send
fibre into another fibre. The fibres generate the vertical distributions,
the horizontal distributions are given by connections. Hence the decomposition
for these flows with respect to these pairs of distributions exists for
all $t\geq 0$.  A standard example in this context is the linearised flow
$\varphi_{t*}: T_{x_0}M \rightarrow T_{\varphi_t(x_0)}$ in the linear frame
bundle $\pi: BM \rightarrow M$. By the Corollary above, the
decomposition $\varphi_{t*}=\xi_t \circ \psi_t$ holds
for all $t\geq 0$ since $\varphi_{t*}$ is an isomorphism between the fibres
$\pi^{-1}(x_0)$  and $\pi^{-1}(\varphi_t(x_0))$.

%
%
%
%
%


\begin{thebibliography}{99}





\bibitem {ap}D. Applebaum -- {\it Lévy processes and stochastic calculus}.
Cambridge University Press, 2004.





\bibitem{Bismut}{J. M.  Bismut} -- Mécanique aléatoire, {\it Lecture Notes in
Math.}, 866. Springer-Verlag, Berlin-New York, 1981.










\bibitem{CSR} P. J. Catuogno, F. B. da Silva and P. R. Ruffino -- Decomposition
of stochastic flows with automorphism of subbundles component. {\it Stochastics
and Dynamics}  12 (2012), no. 2, 1150013, 15 pp.

\bibitem{CSR2} P. J. Catuogno, F. B. da Silva and P. R. Ruffino --
Decomposition of stochastic flows in manifolds with complementary distributions.
{\it Stochastics and Dynamics}  13 (2013), no. 4, 1350009, 12 pp.

\bibitem{Colonius-Kliemann} F. Colonius and W. Kliemann -- {\it The Dynamics of
Control}. Systems and Control:
Foundations and Applications. Birkhäuser Boston, Inc., Boston, MA, 2000.

\bibitem{Colonius-Ruffino} F. Colonius and P. R. Ruffino -- Nonlinear Iwasawa
decomposition of control flows.  {\it Discrete Continuous Dyn. Systems}, (Serie
A) Vol. 18 (2) p. 339-354, 2007.











\bibitem{Gargate-Ruffino} I. G. Gargate and P. R. Ruffino -- An averaging
principle for diffusions in foliated spaces. To appear in {\it Annals of
Probab.} 2015.


\bibitem{Kunita-1}  H. Kunita -- Stochastic differential equations and
stochastic flows of diffeomorphisms,  in:  \'{E}cole d'Et\'{e} de
Probabilit\'{e}s de Saint-Flour XII - 1982, ed. P.L. Hennequin, pp. 143--303.
{\it Lecture Notes on Maths.} 1097 (Springer,1984).

\bibitem{Kunita-2}  H. Kunita -- {\it Stochastic flows and stochastic
differential equations}, Cambridge University Press, 1988.


\bibitem{Li} X.-M. Li -- An averaging principle for a completely integrable
 stochastic Hamiltonian system. {\it Nonlinearity} 21 (2008), no. 4, 803-822.



\bibitem {KPP} T. Kurtz, E. Pardoux and P. Protter -- Stratonovich stochastic
differential equations driven by general semimartingales. {\it Annales de l'I.H.P.},
section B, tome 31, n. 2, 1995. p. 351-377.


\bibitem{ML} M. Liao -- { Decomposition of stochastic flows and Lyapunov
exponents}, {\it Probab. Theory Rel. Fields} {117} (2000) 589-607.




\bibitem{Milnor} J. Milnor --  Remarks on infinite dimensional Lie groups. In:
 {\it Relativity, Groups and Topology II}, Les Houches Session XL, 1983. North
 Holland, Amsterdam, 1984.

\bibitem{Neeb} Neeb, K-H -- Infinite Dimensional Lie Groups. Monastir Summer
 Schoool, 2009. Available from
 hal.archives-ouvertes.fr/docs/00/39/.../CoursKarl-HermannNeeb.pdf


\bibitem{Omori}  Omori, H. -- {\it Infinite dimensional Lie groups}.
Translations of Math Monographs 158, AMS, 1997.


\bibitem{MR}  L. Morgado and P. Ruffino -- Extension of time for decomposition
of stochastic flows in spaces with complementary foliations.  {\it
Electron. Comm. Probab.}, vol. 20 (2015), no. 38.

\bibitem{Ruffino} P. Ruffino -- Decomposition of stochastic flow and
rotation matrix. {\it  Stochastics and Dynamics}, Vol. 2 (2002) 93-108.

\bibitem {sanmartin}  M. Patrão and L.A.B. San Martin -- Semiflows on
topological spaces: chain transitivity and semigroups. {\it J. Dynam.
Differential Equations} 19 (2007), no. 1, 155--180.

\bibitem{Twardowska} K. Twardowska -- Approximation theorems of Wong-Zakai
type for stochastic differential equations in infinite dimensions. {\it
 Dissertationes Math.} (Rozprawy Mat.)  325  (1993), 54 pp.


\end{thebibliography}
\end{document}